\documentclass[11pt,oneside,english]{amsart}
\usepackage[T1]{fontenc}
\usepackage[latin9]{inputenc}
\usepackage{geometry}
\usepackage{amstext}
\usepackage{amsthm}
\usepackage{amssymb}
\usepackage{setspace}
\makeatletter
\numberwithin{equation}{section}
\numberwithin{figure}{section}
\theoremstyle{plain}
\newtheorem{thm}{\protect\theoremname}
  \theoremstyle{plain}
  \newtheorem{lem}[thm]{\protect\lemmaname}
  \theoremstyle{plain}
  \newtheorem{cor}[thm]{\protect\corollaryname}
  \theoremstyle{plain}
  \newtheorem{prop}[thm]{\protect\propositionname}
  \theoremstyle{remark}
  \newtheorem{rem}[thm]{\protect\remarkname}
  \theoremstyle{plain}
  \newtheorem*{thm*}{\protect\theoremname}
  \theoremstyle{definition}
  \newtheorem{problem}[thm]{\protect\problemname}
  \theoremstyle{definition}
  \newtheorem*{problem*}{\protect\problemname}

\makeatother

\usepackage{babel}
  \providecommand{\corollaryname}{Corollary}
  \providecommand{\lemmaname}{Lemma}
  \providecommand{\problemname}{Problem}
  \providecommand{\propositionname}{Proposition}
  \providecommand{\remarkname}{Remark}
  \providecommand{\theoremname}{Theorem}
\providecommand{\theoremname}{Theorem}

\begin{document}

\title{Proving Ergodicity via divergence of ergodic sums }

\author{Zemer Kosloff}

\address{Einstein Institute of Mathematics, Edmond J. Safra Campus (Givat
Ram), The Hebrew University, Jerusalem 91904, Israel.}

\email{zemer.kosloff@mail.huji.ac.il}

\thanks{This research was partially supported by ISF grant No. 1570/17. We
would like to thank Raimundo Briceno for his valuable comments on
the first draft of the paper which greatly improved the presentation
and Emmanuel Roy who stimulated our interest in Poisson suspenisons. }
\begin{abstract}
A classical fact in ergodic theory is that ergodicity is equivalent
to almost everywhere divergence of ergodic sums of all nonnegative
integrable functions which are not identically zero. We show two methods,
one in the measure preserving case and one in the nonsingular case,
which enable one to prove this criteria by checking it on a dense
collection of functions and then extending it to all nonnegative functions.
The first method, Theorem \ref{thm: simple ergodicity criteria},
is then used in a new proof of a folklore criterion for ergodicity
of Poisson suspensions which does not make any reference to Fock spaces.
The second method, Theorem \ref{thm: nonsingular erg}, which involves
the double tail relation is used to show that a large class of nonsingular
Bernoulli and inhomogeneous Markov shifts are ergodic if and only
if they are conservative. In the last section we discuss an extension
of the Bernoulli shift result to other countable groups including
$\mathbb{Z}^{d},\ d\geq2$ and discrete Heisenberg groups. 
\end{abstract}

\maketitle

\section{Introduction}

Given a non-singular system $\left(X,\mathcal{B},\mu,T\right)$, one
of the major challenges is to prove ergodicity. In the finite measure
preserving case, a common approach is to establish that for every
$f\in L^{1}\left(X,\mu\right)$, for $\mu-$a.e. $x\in X$,
\[
\frac{1}{n}\sum_{k=0}^{n-1}f\circ T^{k}(x)=\frac{1}{n}S_{n}(f)(x)\xrightarrow[n\to\infty]{}\int_{X}fd\mu.
\]
It then follows from the pointwise ergodic theorem that $T$ is ergodic.
The maximal inequality, which states that there exists $C>0$, such
that for all $f\in L^{1}\left(X,\mu\right)$ and $t>0$,
\[
\mu\left(\sup_{n\in\mathbb{N}}\left|\frac{S_{n}(f)}{n}\right|>t\right)\leq C\frac{|f|_{L^{1}}}{t}
\]
is used in the classical proof of the pointwise ergodic theorem in
order to establish the almost everywhere convergence for all $f\in L^{1}\left(X,\mu\right)$
from the knowledge of a.e. convergence for a dense set of integrable
functions. This principle lies in the heart of the Hopf method, which
is a method of proving ergodicity for many smooth systems by showing
that for s dense collection of continuous functions $f$, 
\[
\lim_{n\to\infty}\frac{S_{n}(f)}{n}=\int fd\mu,\ \ \mu-a.e.,
\]
See \cite{Wil12} and references therein for a description of the
Hopf method and some references of where it has been used for proving
ergodicity. In the case of infinite $\sigma$-finite measure preserving
systems, one can replace the pointwise ergodic theorem with Hopf's
ratio ergodic theorem by fixing a well chosen positive integrable
function $g\in L^{1}\left(X,\mu\right)$ and then showing that for
all $f\in L^{1}\left(X,\mu\right)$, 
\[
\frac{S_{n}(f)}{S_{n}(g)}\xrightarrow[n\to\infty]{}\frac{\int_{X}fd\mu}{\int_{X}gd\mu},\ \ \mu-a.e.
\]
Again by a maximal inequality it is enough to establish this convergence
for a dense class of $f$ in $L^{1}\left(X,\mu\right)$ and this is
the starting point in Coudene's extension of the Hopf method for some
infinite measure preserving systems \cite{Cou07}. See also \cite{Rob03,Sch16}
for other cases where ergodicity is proved via the ratio ergodic theorem.
A similar method can be done in the case of non-singular systems by
replacing the ratio ergodic theorem with Hurewicz ergodic theorem.
Indeed, this is used in \cite{Kren70Gen,SilThieu95skewent} to show
that a non-singular $K$-system is ergodic if and only if it is conservative. 

Another criteria for ergodicity is that for every $0\leq f\in L^{1}\left(X,\mu\right)$
with $\int_{X}fd\mu>0$, \footnote{It is enough to consider $\left\{ f=1_{A}:\ A\in\mathcal{B},\ 0<\mu(A)<\infty\right\} .$ }
\[
S_{n}(f)\xrightarrow[n\to\infty]{}\infty,\ \ \mu-a.e.
\]

In this note we first make use of this well known ergodicity criteria
for proving ergodicity in two cases, namely Poisson suspensions and
(in-homogenous) Markov shifts. Given a standard probability space
$\left(X,\mathcal{B},\mu\right)$, we say that a collection of sets
$\mathcal{A}\subset\mathcal{B}$ is dense in $\mathcal{B}$ if for
every $B\in\mathcal{B}$ and $\epsilon>0$ there exists $A\in\mathcal{A}$
with $\mu\left(A\right)>0$ and 
\[
\mu\left(A\cap B\right)\geq(1-\epsilon)\mu(A).
\]
 
\begin{thm}
\label{thm: simple ergodicity criteria}Let $\left(X,\mathcal{B},\mu\right)$
be a standard probability space, $T:X\to X$ a measure preserving
system, and $\mathcal{A}\subset\mathcal{B}$ a collection of sets
which is dense in $\mathcal{B}$. If there exists $\alpha>0$ such
that for all $A\in\mathcal{A}$ there exists two subsequences $n_{j}\to\infty$
and $N_{n}\to\infty$ such that 
\[
\liminf_{n\to\infty}\left(\frac{1}{N_{n}}\sum_{k=0}^{N_{n}-1}1_{A}\circ T^{n_{k}}\right)\geq\alpha\mu(A),\ \ \mu-a.e.,
\]
then $T$ is ergodic. 
\end{thm}
We use this criteria to show a new proof of the classical fact that
a measure preserving Poisson suspension $\left(X^{*},\mathcal{B}^{*},\mu^{*},S_{*}\right)$
is ergodic if and only if there exists no absolutely continuous invariant
probability measure for $\left(X,\mathcal{B},\mu,S\right)$. Our proof
does not involve any use of the Fock Space structure and therefore
it might be useful for proving ergodicity of other more complicated
point processes. We also show that a simple argument shows that $\left(X^{*},\mathcal{B}^{*},\mu^{*},S_{*}\right)$
is indeed weak mixing if and only if it is ergodic, thus this direct
method recovers the full statement of the classical fact. 

In the case where $T$ is invertible and nonsingular, $T$ is ergodic
if and only if for all $0\leq f\in L^{1}\left(X,\mu\right)$ with
$\int_{X}fd\mu>0$, 
\[
\hat{T}_{n}(f)=\sum_{k=0}^{n-1}\frac{d\left(\mu\circ T^{-k}\right)}{d\mu}f\circ T^{-k}\xrightarrow[n\to\infty]{}\infty,\ \ \mu-a.e.
\]

Given a finite or countable set $F$, a closed shift invariant subset
$X\subset F^{\mathbb{Z}}$ is called a subshift. The double tail relation
of $X$ is a Borel subset of $X\times X$ defined by 
\[
\mathcal{T}=\left\{ (x,y)\in X\times X:\ \exists n\in\mathbb{N},\ \forall|k|>n,\ x_{k}=y_{k}\right\} .
\]

The set $\mathcal{T}$ is in addition an equivalence relation on $X$
. Given a probability measure $\mu$ on $X$ for which the shift $T$
is non-singular, the symbolic system $\left(X,\mathcal{B},\mu,T\right)$
is double tail trivial if for all $A\in\mathcal{B}$, $\mu\left(\mathcal{T}\left(A\right)\right)=0$
or $\mu\left(X\mathcal{\backslash T}\left(A\right)\right)=0$, where
\[
\mathcal{T}\left(A\right)=\cup_{x\in A}[x]_{\sim}=\left\{ y\in X:\ \exists x\in A,\ (x,y)\in\mathcal{T}\right\} .
\]
See \cite{PetSch97} and the references therein for examples of double
tail trivial processes. 
\begin{thm}
\label{thm: nonsingular erg}Let $\left(X,\mathcal{B},\mu,T\right)$
be a conservative, non-singular subshift which is double tail trivial
and $\mathcal{A}\subset\mathcal{B}$ the collection of finite union
of cylinder sets in $\mathcal{B}$. If there exists $L:\mathcal{T}\to\left(0,\infty\right)$
such that for $\mu\times\mu$ almost all $(x,y)\in\mathcal{T}$, for
all $n\in\mathbb{N}$, 
\[
L(x,y)^{-1}\frac{d\left(\mu\circ T^{-n}\right)}{d\mu}(y)\leq\frac{d\left(\mu\circ T^{-n}\right)}{d\mu}(x)\leq L(x,y)\frac{d\left(\mu\circ T^{-n}\right)}{d\mu}(y),
\]
then $T$ is ergodic. 
\end{thm}
In Section \ref{sec:Examples} we use this theorem to show an ergodicity
criterion for two natural symbolic models, non-singular Bernoulli
shifts which are shifts of independent not necessarily identically
distributed random variables and in-homogenous Markov (chains) shifts
which are fully supported on a topologically mixing subshift of finite
type. We give a short discussion on how the latter implies a certain
hurdle for a natural approach towards a variant on a classical question
of Bowen on the existence of a measure preserving $C^{1}$ Anosov
diffeomorphism of $\mathbb{T}^{2}$ which is not ergodic. Finally
we extend the result on Bernoulli shifts for countable groups which
have a version of the Hurewicz's ratio ergodic theorem. 

We end the introduction with a description of the result in the case
of non-singular Bernoulli shifts. A non singular Bernoulli shift is
a quadruple $\left(\left\{ 1,..,N\right\} ^{\mathbb{N}},\mathcal{B},\mu,T\right)$
where $\mu=\prod_{k\in\mathbb{Z}}\mu_{k}$ is a product measure on
$\{1,..,N\}^{\mathbb{Z}}$, $T$ is the shift map on $\left\{ 1,...,N\right\} ^{\mathbb{Z}}$
defined by 
\[
\left(Tx\right)_{i}=x_{i+1}
\]
and $\mu\sim\mu\circ T$ (i.e. the shift is $\mu$- non-singular).
By Kakutani's theorem, non-singularity of the shift is equivalent
to
\begin{equation}
\sum_{k\in\mathbb{Z}}\sum_{j=1}^{N}\left(\sqrt{\mu_{k}(j)}-\sqrt{\mu_{k-1}\left(j\right)}\right)^{2}<\infty.\label{eq: Kakutani}
\end{equation}
 In the case where in addition there exists a probability distribution
$P$ on $\left\{ 1,..,N\right\} $ such that $\mu_{k}=P$ for all
$k<0$, the Bernoulli shift is a $K$-automorphism in the sense of
Silva and Thieullen \cite{SilThieu95skewent}, hence it is ergodic
if and only if it is conservative. We prove that ergodicity is equivalent
to conservativity for general, not necessarily half stationary, Bernoulli
shifts satisfying a natural non-degeneracy condition.
\begin{thm}
\label{thm: ergodicity of Bernoulli shifts}If a non-singular Bernoulli
shift $\left(\left\{ 1,..,N\right\} ^{\mathbb{N}},\mathcal{B},\prod\mu_{k},T\right)$
is conservative and 
\begin{equation}
L=\sup_{k\in\mathbb{Z}}\frac{\max_{j\in\{1,...,N\}}\left(\mu_{k}\left(\{j\}\right)\right)}{\min_{j\in\{1,...,N\}}\left(\mu_{k}\left(\{j\}\right)\right)}<\infty,\label{eq: condition on B-shifts}
\end{equation}
 then it is ergodic. 
\end{thm}
Notation:
\begin{itemize}
\item $a_{n}\lesssim b_{n}$ means that $\limsup_{n\to\infty}\frac{a_{n}}{b_{n}}\leq1$. 
\item For $a,b\grave{in}\mathbb{R}$ and $c>0$ we write $a=b\pm c$ if
$\left|a-b\right|<c$. 
\item For $a,b>0$ and $L>1$, $a=bL^{\pm\epsilon}$ means $bL^{-\epsilon}\leq a\leq bL^{\epsilon}$. 
\end{itemize}

\section{Proof of Theorem \ref{thm: simple ergodicity criteria}}

Suppose $T$ is not ergodic, then there exists a $T$ invariant set
$B\in\mathcal{B}$ with $\mu(B),\ \mu\left(X\backslash B\right)>0$.
As $\mathcal{A}$ generates $\mathcal{B}$ there exists $A\in\mathcal{A}$
with 
\[
\mu\left(A\cap B\right)\geq\left(1-\alpha\mu\left(X\backslash B\right)\right)\mu(A).
\]
By the assumptions of Theorem \ref{thm: simple ergodicity criteria},
there exists $n_{j}\to\infty$ and $N_{j}\to\infty$ such that for
$\mu$ almost every $x\in X$, 
\[
\liminf_{n\to\infty}\left(\frac{1}{N_{n}}\sum_{k=0}^{N_{n}-1}1_{A}\circ T^{n_{k}}(x)\right)\geq\alpha\mu(A).
\]
In addition, by $T$ invariance of $B$, for every $x\in X\backslash B$
and $n\in\mathbb{N}$
\begin{equation}
\sum_{k=0}^{n-1}1_{A}\circ T^{n_{k}}(x)=\sum_{k=0}^{n-1}1_{A\backslash B}\circ T^{n_{k}}(x).\label{eq: bbb}
\end{equation}

By Fatou's lemma,
\begin{align*}
\alpha\mu(A)\mu(X\backslash B) & \leq\int_{X\backslash B}\liminf_{n\to\infty}\left(\frac{1}{N_{n}}\sum_{k=0}^{N_{n}-1}1_{A}\circ T^{n_{k}}(x)\right)d\mu\\
 & \leq\liminf_{n\to\infty}\int_{X\backslash B}\left(\frac{1}{N_{n}}\sum_{k=0}^{N_{n}-1}1_{A}\circ T^{n_{k}}(x)\right)d\mu\\
 & \overset{\eqref{eq: bbb}}{\leq}\liminf_{n\to\infty}\int_{X\backslash B}\left(\frac{1}{N_{n}}\sum_{k=0}^{N_{n}-1}1_{A\backslash B}\circ T^{n_{k}}(x)\right)d\mu\\
 & \leq\mu(A\backslash B)<\alpha\mu(A)\mu\left(X\backslash B\right)
\end{align*}

This is a contradiction, hence $T$ is ergodic. 

\section{Folklore criteria for ergodicity of Poisson suspensions}

Let $\left(X,\mathcal{B},\mu\right)$ be a standard $\sigma$-finite
measure space and $\left(X^{*},\mathcal{B}^{*},\mu^{*}\right)$ its
associated Poisson point process. That is, $X^{*}$ is the collection
of all countable subsets of $X$ (or counting measures), $\mathcal{B}^{*}$
the $\sigma$-algebra generated by 
\[
\left\{ \nu\in X^{*}:\ N(A)(\nu)=n\right\} 
\]
with $A\in\mathcal{B}$ with $0<\mu(A)<\infty$ and $n\in\mathbb{N}\cup\{0,\infty\}$,
where 

\[
N(A)(\nu)=|\nu\cap A|.
\]
 Finally, the measure $\mu^{*}$ is the unique measure such that for
all pairwise disjoint sets $A_{1},A_{2},..,A_{n}\in\mathcal{\mathcal{B}}$,
the random variables $\left\{ N\left(A_{i}\right)\right\} _{i=1}^{n}$
are independent and for each $A\in\mathcal{B}$ with $\mu(A)<\infty$,
$N(A)$ is Poisson distributed with parameter $\mu(A)$, that is,
for all $k\in\mathbb{N\cup}\{0\}$,
\[
\mu^{*}\left(N(A)=k\right)=\frac{e^{-\mu(A)}\mu(A)^{k}}{k!}.
\]
Given a measure preserving transformation $T:\left(X,\mathcal{B},\mu\right)\to\left(X,\mathcal{B},\mu\right)$,
its Poisson suspension is a probability preserving map $T_{*}:\left(X^{*},\mathcal{B}^{*},m^{*}\right)\to\left(X^{*},\mathcal{B}^{*},m^{*}\right)$
defined by 
\[
T_{*}\left(\left\{ x\right\} _{x\in\nu}\right)=\left\{ Tx\right\} _{x\in\nu}.
\]

In what follows we will write 
\[
\mathcal{B}_{fin}=\left\{ B\in\mathcal{B}:\ 0<\mu(B)<\infty\right\} .
\]

\begin{thm}
\label{thm: Ergodicity of P-susp} Let $\left(X,\mathcal{B},\mu,T\right)$
be a $\sigma$-finite measure preserving system. Then $T_{*}$ is
ergodic if and only if $T$ has no absolutely continuous invariant
probability measure.
\end{thm}
If $T$ has an absolutely continuous invariant probability (a.c.i.p.),
then it is immediate that $T_{*}$ is not ergodic as in that case
there is a set $A\in\mathcal{B}_{fin}$ with $T^{-1}A=A$, For all
$k\in\mathbb{N}$ the sets 
\[
\left[N(A)=K\right]=\left\{ \nu\in X^{*}:\ N(A)(\nu)=k\right\} 
\]
are $T_{*}$ invariant sets of positive, non-full $\mu^{*}$-measure.
Our proof of ergodicity of $T_{*}$ when $T$ has no a.c.i.p. is by
establishing the conditions of Theorem \ref{thm: simple ergodicity criteria}
with the collection of sets 
\[
\mathcal{A}^{*}=\left\{ \bigcap_{i=1}^{L}\left[N\left(A_{i}\right)=k_{i}\right]:\ L\in\mathbb{N},\ \left\{ A_{i}\right\} _{i=1}^{L}\subset\mathcal{B}_{fin}\ (\text{pairwise disjoint)},\ \left\{ k_{i}\right\} _{i=1}^{L}\subset\mathbb{N}\cup\{0\}\right\} 
\]
 and $\alpha=1$. 
\begin{lem}
\label{lem: Null part}Let $\left(X,\mathcal{B},\mu,T\right)$ be
a $\sigma$-finite measure preserving system. If there exists no absolutely
continuous invariant probability measure, then for all $L\in\mathbb{N},$
$A_{1},A_{2},..A_{L}\in\mathcal{B}_{fin},$ there exists a strictly
increasing subsequence $n_{k}\to\infty$ such that for all $\alpha,\beta\in\left\{ 1,2,...,L\right\} $,
\[
\lim_{j-l\to\infty}\mu\left(A_{\alpha}\cap T^{-\left(n_{j}-n_{l}\right)}A_{\beta}\right)=0.
\]
\end{lem}
A subset $K\subset\mathbb{N}$ has \textit{full Banach density} if
\[
\lim_{n\to\infty}\frac{\left|K\cap[1,n]\right|}{n}=1.
\]
In what follows we will use the well known fact that if $a_{n}\geq0$
satisfies 
\[
\lim_{n\to\infty}\frac{\sum_{i=1}^{n}a_{i}}{n}=0,
\]
then for all $\epsilon>0$, the sequence $K^{\epsilon}=\left\{ n\in\mathbb{N}:\ 0\leq a_{n}<\epsilon\right\} $
has full Banach density. Another trivial consequence of the definition
of full Banach density is that if $K_{1},K_{2},..,K_{N}\subset\mathbb{N}$
are sets of full Banach density then $\bigcap_{i=1}^{N}K_{i}$ has
full Banach density.
\begin{proof}
Let $A_{1},A_{2},...,A_{L}\in\mathcal{B}_{fin}$. We construct $n_{k}\to\infty$
by an inductive procedure. As $T$ is $\mu$ measure preserving and
there exists no a.c.i.p., given a finite set $F\subset\mathbb{N}$
and $\left\{ B_{\alpha}\right\} _{\alpha\in F}\subset\mathcal{B}_{fin}$
, by the pointwise ergodic theorem, for all $\alpha\in F$, 
\[
\frac{1}{n}S_{n}\left(1_{B_{\alpha}}\right)\xrightarrow[n\to\infty]{}0,\ \ \mu-\text{a.e.}
\]

By the dominated convergence theorem for all $\alpha,\beta\in F$,
\[
\frac{1}{n}\sum_{k=0}^{n-1}\mu\left(A_{\alpha}\cap T^{-k}A_{\beta}\right)=\int_{A_{\alpha}}\left(\frac{1}{n}S_{n}\left(1_{A_{\beta}}\right)\right)d\mu\xrightarrow[n\to\infty]{}0.
\]
We conclude, using the previous discussion on sets of full Banach
density, that for all $\epsilon>0$, the set 
\[
K:^{\epsilon}=\left\{ n\in\mathbb{N}:\ \forall\alpha,\beta\in F,\ \mu\left(A_{\alpha}\cap T^{-n}A_{\beta}\right)<\epsilon\right\} 
\]
is of full Banach density. 

Taking first $F=\left\{ 1,..,L\right\} $ and for $\alpha\in F$,
$B_{\alpha}=A_{\alpha}$, we can choose $n_{1}\in\mathbb{N}$ such
that for all $\alpha,\beta\in\left\{ 1,...,L\right\} $, 
\[
\mu\left(A_{\alpha}\cap T^{-n_{1}}A_{\beta}\right)<\frac{1}{2}.
\]
Assume that we have chosen a sequence $n_{0}=0$ and $n_{1},..,n_{k}\in\mathbb{N}$
such that for all $0\leq l<j\leq k$ and $\alpha,\beta\in F$, 
\[
\mu\left(T^{-n_{l}}A_{\alpha}\cap T^{-n_{j}}A_{\beta}\right)=\mu\left(A_{\alpha}\cap T^{-\left(n_{j}-n_{l}\right)}A_{\beta}\right)<2^{-j}.
\]
 Looking at $F_{k}=\left\{ 1,2,..,kL\right\} $ and 
\[
B_{s}=T^{-n_{j}}A_{\alpha}\ \ \text{for }s=jL+\alpha,
\]
we conclude that the set 
\[
\left\{ n\in\mathbb{N}:\ \forall\alpha,\beta\in F_{k},\ \mu\left(B_{\alpha}\cap T^{-n}B_{\beta}\right)<2^{-(k+1)}\right\} 
\]
is of full Banach density. In particular there exists $n_{k+1}>n_{k},$
such that for all $\alpha,\beta\in\{1,..,L\}$, and $0\leq j<k+1$
\[
\mu\left(T^{-n_{j}}A_{\alpha}\cap T^{-n_{k+1}}A_{\beta}\right)=\mu\left(B_{\left(jL+\alpha\right)}\cap T^{-n_{k+1}}B_{\beta}\right)<2^{-(k+1)}
\]
as was required. This concludes the proof of the lemma. 
\end{proof}
Given $B=\bigcap_{j=1}^{L}\left[N\left(A_{j}\right)=k_{j}\right]$
a (Poissonian) cylinder set we write $\mathbf{S}(B)=\bigcup_{i=1}^{L}A_{j}$. 
\begin{lem}
\label{lem: mixing of Poisson}If $B,C\in\mathcal{B}^{*}$ are cylinder
sets then 
\[
\left|\mu^{*}\left(B\cap C\right)-\mu^{*}\left(B\right)\mu^{*}\left(C\right)\right|\leq2\mu\left({\bf S}\left(B\right)\cap{\bf S}\left(C\right)\right).
\]
\end{lem}
\begin{proof}
Note that as $\mu^{*}$ is a probability measure we can assume that
$\mu\left({\bf S}\left(B\right)\triangle{\bf S}\left(C\right)\right)<1$.
Write $B=\bigcap_{j=1}^{L}\left[N\left(A_{j}\right)=k_{j}\right]$
and $D=\bigcap_{j=1}^{L}\left[N\left(A_{j}\backslash{\bf S}\left(C\right)\right)=k_{j}\right]$.
Note that 
\[
B\triangle D\subset\left[N\left({\bf S}\left(B\right)\cap{\bf S}\left(C\right)\right)>0\right].
\]
Thus 
\begin{align*}
\mu^{*}\left(B\triangle D\right) & \leq1-\mu^{*}\left(N\left({\bf S}\left(B\right)\cap{\bf S}\left(C\right)\right)=0\right)\\
 & =1-\exp\left(-\mu\left({\bf S}\left(B\right)\cap{\bf S}\left(C\right)\right)\right)\leq\mu\left({\bf S}\left(B\right)\cap{\bf S}\left(C\right)\right).
\end{align*}
As ${\bf S}\left(C\right)\cap{\bf S}\left(D\right)=\emptyset$ by
the independence property of the Poisson process,
\[
\mu^{*}\left(D\cap C\right)=\mu^{*}\left(D\right)\mu^{*}\left(C\right)=\mu^{*}\left(B\right)\mu^{*}\left(C\right)\pm\mu\left({\bf S}\left(B\right)\cap{\bf S}\left(C\right)\right).
\]
Similarly 
\[
\left|\mu^{*}\left(B\cap C\right)-\mu^{*}\left(D\cap C\right)\right|\leq\mu^{*}\left(B\triangle D\right).
\]
This shows that
\[
\left|\mu^{*}\left(B\cap C\right)-\mu^{*}\left(B\right)\mu^{*}\left(C\right)\right|\leq2\mu\left({\bf S}\left(B\right)\cap{\bf S}\left(C\right)\right).
\]
\end{proof}
\begin{cor}
\label{cor: mixing}For all $A_{1},A_{2},..,A_{L}\in\mathcal{B}_{fin}$
and $k_{1},k_{2},..,k_{L}\in\mathbb{N}\cup\{0\}$, there exists a
subsequence $n_{j}\to\infty$ such that writing $B=\bigcap_{j=1}^{L}\left[N\left(A_{j}\right)=k_{j}\right]\in\mathcal{B}^{*}$,
for all $0\leq l<l$,
\[
\left|\mu^{*}\left(T_{*}^{-n_{l}}B\cap T_{*}^{-n_{j}}B\right)-\mu^{*}\left(B\right)^{2}\right|\leq2^{-\left(j-l\right)}.
\]
\end{cor}
\begin{proof}
For all cylinder sets $B$ and $n\in\mathbb{Z},$ 
\[
{\bf S}\left(T_{*}^{n}B\right)=T^{n}{\bf S}\left(B\right).
\]
By Lemma \ref{lem: Null part} there exists $n_{j}\to\infty$ such
that for all $1\leq l<j$ , 
\begin{align*}
\mu\left({\bf S}\left(T_{*}^{-n_{j}}B\right)\cap{\bf S}\left(T_{*}^{-n_{l}}B\right)\right) & =\mu\left(T^{-n_{j}}{\bf S}\left(B\right)\cap T^{-n_{l}}{\bf S}\left(B\right)\right)\\
 & =\mu\left({\bf S}\left(B\right)\cap T^{-\left(n_{j}-n_{l}\right)}{\bf S}\left(B\right)\right)\leq2^{-\left(j+1-l\right)}\mu^{*}.
\end{align*}
The conclusion follows from Lemma \ref{lem: mixing of Poisson} as
$T_{*}$ is $\mu^{*}$ preserving.
\end{proof}
\begin{proof} [Proof of \ref{thm: Ergodicity of P-susp}]

Note that as $T$ preserves $\mu$, if $T$ has an a.c.i.p., then
there exists a set a set $A\in\mathcal{B}$ such that $T^{-1}A=A\mod\mu$
and $\mu\left(A\right)<\infty$. In that case for each $K\in\mathbb{N}$,
the set $\left[N(A)=k\right]$ is $T$ invariant and of positive measure.
As for $K\neq K'$, 
\[
\left[N(A)=K\right]\cap\left[N(A)=K'\right]=\emptyset,
\]
 this is a contradiction to ergodicity. 

In the other direction assume $T$ has no absolutely continuous invariant
probability measure. The collection $\mathcal{A}^{*}$ generates $\mathcal{B}^{*}.$
We show that the conditions of Theorem \ref{thm: simple ergodicity criteria}
hold for all $B\in\mathcal{A}^{*}$. Let $B\in\mathcal{A}^{*}$. By
Corollary \ref{cor: mixing}, there exists a sequence $n_{j}\to\infty$
such that for all $l<j$,
\[
\int_{X^{*}}1_{B}\circ T_{*}^{n_{j}}1_{B}\circ T_{*}^{n_{l}}d\mu^{*}=\mu^{*}\left(B\cap T_{*}^{-\left(n_{j}-n_{l}\right)}B\right)\leq\left(1+2^{-\left(j-l\right)}\right)\mu^{*}\left(B\right)^{2}.
\]
By this, for all $N\in\mathbb{N}$, 
\begin{align*}
\int_{X^{*}}\left(\sum_{j=0}^{N-1}1_{B}\circ T_{*}^{n_{j}}\right)^{2}d\mu^{*} & =\sum_{j=0}^{N-1}\int_{X^{*}}1_{B}\circ T_{*}^{n_{j}}d\mu^{*}+2\sum_{0\leq l<j<N}\int_{X^{*}}1_{B}\circ T_{*}^{n_{j}}1_{B}\circ T_{*}^{n_{l}}d\mu^{*}\\
 & =N\mu^{*}\left(B\right)+2\sum_{0\leq l<j<N}\left(1+2^{-\left(j-l\right)}\right)\mu^{*}\left(B\right)^{2}\\
 & =N^{2}\left(\mu^{*}\left(B\right)\right)^{2}+O(N)\\
 & =\left(\int_{X^{*}}\sum_{j=0}^{N-1}1_{B}\circ T_{*}^{n_{j}}d\mu^{*}\right)^{2}+O(N).
\end{align*}
This shows that 
\[
Var\left(\frac{1}{N}\sum_{j=0}^{N-1}1_{B}\circ T_{*}^{n_{j}}\right)=O\left(\frac{1}{N}\right)\xrightarrow[N\to\infty]{}0.
\]
A classical application of Chebychev's inequality then shows that
\[
\frac{1}{N}\sum_{j=0}^{N-1}1_{B}\circ T_{*}^{n_{j}}\xrightarrow{}\mu^{*}\left(B\right),\ \text{in}\ \mu^{*}\ \text{measure.}
\]
It then follows that there exists $N_{n}\to\infty$ such that 
\[
\frac{1}{N_{n}}\sum_{j=0}^{N_{n}-1}1_{B}\circ T_{*}^{n_{j}}\xrightarrow[n\to\infty]{}\mu^{*}\left(B\right),\ \mu^{*}-\text{almost everywhere.}
\]
We have shown that for all $B\in\mathcal{A}^{*}$ we have $n_{j}\to\infty$
and $N_{n}\to\infty$ as in the conditions of Theorem \ref{thm: simple ergodicity criteria}
and therefore $T_{*}$ is ergodic. 

\end{proof}

\subsubsection{Ergodicity implies weak mixing}

For $T_{*}$ a Poisson suspension over a measure preserving transformation
$T$, it is known that if $T_{*}$ is ergodic then $T_{*}$ is weak
mixing. We show an argument which gives the weak mixing result. 

A probability preserving transformation $\left(\Omega,\mathcal{C},\nu,R\right)$
is weak mixing if $R\times R$ is ergodic. Weak mixing implies ergodicity
and there are several (standard) equivalent definitions of the weak
mixing property. Among them is the spectral condition, $R$ is weakly
mixing if and only if there are no functions $f\in L^{2}\left(\Omega,\nu\right)$
with $\int fd\nu=0$ and $\lambda\in\mathbb{C}$ with $|\lambda|=1$
which satisfy
\[
f\circ R=\lambda f.
\]

\begin{prop}
\label{prop: Aaronson}\footnote{This proposition was communicated to us by J. Aaronson.}Let
$\left(\Omega,\mathcal{C},\nu,R\right)$ be a probability preserving
transformation. If for all $f\in L^{2}\left(\Omega,\nu\right)$, there
exists $n_{j}\to\infty$ such that 
\[
\int f\circ R^{n_{j}}\bar{f}d\nu\xrightarrow[j\to\infty]{}\left|\int fd\nu\right|^{2}
\]
 as $j\to\infty$, then $R$ is weak mixing. 
\end{prop}
\begin{proof}
Assume that the conditions of the proposition are satisfied and $R$
is not weak mixing. Then there exists a non constant $f\in L^{2}\left(\Omega,\nu\right)$
with $\int fd\nu=0$ and $\lambda\in\mathbb{C}$ with $|\lambda|=1$
such that 
\[
f\circ R=\lambda f.
\]
By the conditions of the proposition, there exists $n_{j}\to\infty$
such that 
\begin{align*}
0 & =\lim_{j\to\infty}\int_{\Omega}f\circ R^{n_{j}}\bar{f}d\nu\\
 & =\lim_{j\to\infty}\lambda^{n_{j}}\int_{\Omega}|f|^{2}d\nu.
\end{align*}
This can happen only if $f\equiv0$, which is a contradiction. 
\end{proof}
As $\left(X,\mathcal{B},\mu\right)$ is a standard $\sigma$-finite
measure space, there exists a countable collection of sets $\mathcal{Z}\subset\mathcal{B}_{fin}$
such that for all $A\in\mathcal{B}_{fin}$ and $\epsilon>0$, there
exists $C$ which is a finite union of sets in $\mathcal{Z}$ such
that $\mu\left(A\triangle C\right)<\epsilon$. In the case $X=\mathbb{R}$
and $\mu$ the Lebsegue measure one can take for example $\mathcal{Z}$
to be the collection of intervals with rational endpoints. Denote
by $\mathcal{F}$ the collection of finite unions of sets in $\mathcal{Z}$
and define
\[
\mathcal{A}\left(\mathcal{Z}\right)=\left\{ \bigcap_{i=1}^{L}\left[N\left(A_{i}\right)=k_{i}\right]:\ L\in\mathbb{N},\ \left\{ A_{i}\right\} _{i=1}^{L}\subset\mathcal{F},\ \left\{ k_{i}\right\} _{i=1}^{L}\subset\mathbb{N}\cup\{0\}\right\} .
\]

In what follows, the fact that $\mathcal{F}$ and hence $\mathcal{A}\left(\mathcal{Z}\right)$
are countable will useful. 
\begin{lem}
\label{lem: density of collection}The collection of simple functions
of the form $\sum_{i=1}^{L}c_{i}{\bf 1}_{A_{i}^{*}}$ with $\left\{ A_{i}^{*}\right\} _{i=1}^{L}\subset\mathcal{A}\left(\mathcal{Z}\right)$
is dense in $L^{2}\left(X^{*},\mu^{*}\right)$. 
\end{lem}
\begin{proof}
Firstly, the collection of simple functions with $\left\{ A_{i}^{*}\right\} _{i=1}^{L}\subset\mathcal{A}^{*}$
is dense in $L^{2}\left(X^{*},\mu^{*}\right)$. 

Secondly, for any $\left\{ A_{i}\right\} _{i=1}^{N}\subset\mathcal{B}_{fin}$
there exists an array of sets $\left\{ B_{n,i}:\ n,i\in\mathbb{N},1\leq i\leq N\right\} $
such that 
\[
\max_{1\leq i\leq N}\mu\left(A_{i}\triangle B_{n,i}\right)\xrightarrow[n\to\infty]{}0.
\]
By this, for any $A^{*}=\bigcap_{i=1}^{N}\left[N\left(A_{i}\right)=k_{i}\right]\in\mathcal{A}^{*}$
writing $B_{n}^{*}=\bigcap_{i=1}^{N}\left[N\left(B_{n,i}\right)=k_{i}\right]$,
one has 
\[
\left|1_{A^{*}}-1_{B_{n}^{*}}\right|_{2}=\mu\left(A^{*}\triangle B_{n}^{*}\right)\xrightarrow[n\to\infty]{}0.
\]
The combination of these two observations proves the claim. 
\end{proof}
\begin{lem}
Let $\left(X,\mathcal{B},m,T\right)$ be a measure preserving transformation
with $m(X)=\infty$ and no absolutely continuous probability measure.
For any countable collection of sets $\mathcal{V}\subset\mathcal{B}_{fin}$,
there exists $n_{j}\to\infty$ such that for all $A,B\in\mathcal{V}$
\begin{equation}
m\left(A\cap T^{-n_{j}}B\right)\xrightarrow[j\to\infty]{}0.\label{eq: lim 0}
\end{equation}
\end{lem}
\begin{proof}
Let $\left\{ A_{i}\right\} _{i=1}^{\infty}$ be an enumeration of
the sets in $\mathcal{V}$. We will construct $n_{j}$ as follows.
As $T$ has no a.c.i.p. and $\left\{ 1_{A_{i}}\right\} _{i=1}^{\infty}\subset L^{1}\left(X,\mathcal{B},m\right)$,
for all $1\leq i,j<\infty$,
\[
\frac{1}{n}\sum_{i=1}^{n}m\left(A_{i}\cap T^{-n}A_{j}\right)\xrightarrow[n\to\infty]{}0.
\]
Consequently for all $\epsilon>0$ and $L\in\mathbb{N}$, the set
\[
D(L,\epsilon)=\left\{ n\in\mathbb{N}:\ \max_{1\leq i,j\leq L}\left(m\left(A_{i}\cap T^{-n}A_{j}\right)\right)<\epsilon\right\} 
\]
 is of Banach density $1$ since it is an intersection of $2^{L}$
elements of full density. A simple inductive construction gives an
increasing subsequence $n_{j}\to\infty$ such that $n_{j}\in D\left(j,2^{-j}\right)$.
The lemma is proven. 
\end{proof}
\begin{thm}
If $\left(X,\mathcal{B},\mu,T\right)$ be a $\sigma$-finite measure
preserving system with no absolutely continuous invariant probability
measure, then $T_{*}$ is weak mixing.
\end{thm}
\begin{proof}
Let $\mathcal{Z\subset B}_{fin},\mathcal{A}\left(\mathcal{Z}\right)\subset\mathcal{B}^{*}$
and $\mathcal{F}$ be as above and $n_{j}\to\infty$ such that for
all $A,B\in\mathcal{F}$\footnote{Recall that $\mathcal{F}$ is countable.},
\begin{equation}
m\left(A\cap T^{-n_{j}}B\right)\xrightarrow[j\to\infty]{}0.\label{eq: lim0}
\end{equation}
 First we show that for all $C,D\in\mathcal{A}\left(\mathcal{Z}\right)$,
\[
\int_{X^{*}}1_{C}\left(1_{D}\circ T_{*}^{n_{j}}\right)dm^{*}\xrightarrow[j\to\infty]{}m^{*}\left(C\right)m^{*}\left(D\right).
\]

Indeed, any $C,D\in\mathcal{A}\left(Z\right)$ are of the form $C=\bigcap_{i=1}^{L}\left[N\left(A_{i}\right)=k_{i}\right]$
and $D=\bigcap_{i=L+1}^{L+M}\left[N\left(A_{i}\right)=k_{i}\right]$
with $L,M\in\mathbb{N}$, $\left\{ k_{i}\right\} _{i=1}^{L+M}\subset\{0\}\cup\mathbb{N}$
and $\left\{ A_{i}\right\} _{i=1}^{L+M}\subset\mathcal{F}$. As $\mathcal{F}$
is closed under finite unions, ${\bf A}=\bigcup_{i=1}^{L+M}A_{i}\in\mathcal{F}$,
thus 
\[
m\left({\bf A}\cap T^{-n_{j}}{\bf A}\right)\xrightarrow[j\to\infty]{}0.
\]
 By Lemma \ref{lem: mixing of Poisson}
\[
\left|\int_{X^{*}}1_{C}\left(1_{D}\circ T_{*}^{n_{j}}\right)dm^{*}-m^{*}\left(C\right)m^{*}\left(D\right)\right|\xrightarrow[j\to\infty]{}0.
\]
 Consequently, by Lemma \ref{lem: density of collection} and standard
approximation arguments, it then follows that for all $F,G\in L^{2}\left(X^{*},m^{*}\right)$,
\[
\int_{X^{*}}F\left(G\circ T_{*}^{n_{j}}\right)dm^{*}\xrightarrow[j\to\infty]{}\left(\int_{X^{*}}Fdm^{*}\right)\left(\int_{X^{*}}Gdm^{*}\right).
\]
By Proposition \ref{prop: Aaronson}, $T_{*}$ is weak mixing. 
\end{proof}
\begin{rem}
Emmanuel Roy has pointed out to us that in the case of Poisson suspensions
weak mixing follows from ergodicity by the following argument. Given
a measure preserving $T:\left(X,\mathcal{B},\mu\right)\to\left(X,\mathcal{B},\mu\right)$
let $S=T\times Id:X\times\{0,1\}\to X\times\{0,1\}$ be a two point
extension of $T$ which preserves $\mu\times\left(\frac{1}{2}\left(\delta_{0}+\delta_{1}\right)\right)$.
Then $S_{*}$ is isomorphic to $T_{*}\times T_{*}$. Also $T$ has
no a.c.i.p. if and only if $S$ has no a.c.i.p. thus ergodicity of
$T_{*}$ implies ergodicity of $S_{*}\cong T_{*}\times T_{*}$. 
\end{rem}

\subsection{Original motivation for the statement of Theorem \ref{thm: simple ergodicity criteria}. }

A set $W$ in $\left(X,\mathcal{B},m\right)$ is weakly wandering
for $T$ if there exists $n_{j}\to\infty$ such that $\left\{ T^{-n_{j}}W\right\} $
are pairwise disjoint. If there exists no a.c.i.p., then $X$ is a
countable disjoint union of weakly wandering sets $\bigcup_{j\in\mathbb{Z}}T^{-n_{j}}W$,
see \cite{Aar97InfErg,HajKakIt72invmeas} for discussion on weakly
wandering sets. As for all weakly wandering set $W$ and $A\in\mathcal{B}$,
$A\cap W$ is weakly wandering, this implies that every finite measure
set can be approximated from within by a finite union of weakly wandering
sets. Here given $W$ weakly wandering with respect to $n_{j}\to\infty$
of positive and finite $\mu$ measure and $k\in\mathbb{N}$, the sequence
\[
Y_{j}:=1_{\left[N(W)=k\right]}\circ T_{*}^{n_{j}}
\]
is a sequence of i.i.d. integrable random variables. By the strong
law of large numbers, 
\[
\frac{1}{n}\sum_{j=0}^{n-1}1_{\left[N(W)=k\right]}\circ T_{*}^{n_{j}}\xrightarrow[n\to\infty]{}\mu^{*}\left(N(W)=k\right),\ \ \mu^{*}-a.s.
\]
Our first attempt was to use this to show the conditions of Theorem
\ref{thm: simple ergodicity criteria}. The problem for doing this
lies in the following: Given $W_{1},W_{2},..,W_{N}$ pairwise disjoint
weakly wandering sets, does there exists $n_{j}\to\infty$ such that
$\bigcup_{i=1}^{N}W_{i}$ is weakly wandering along $n_{j}$? 

\section{Proof of theorem \ref{thm: nonsingular erg}}

\subsection{Some relevant material from non-singular ergodic theory}

This subsection contains several classical statements and definitions
from non-singular ergodic theory. The reader is referred to \cite{Aar97InfErg}
where the statements and their proofs are written. Let $\left(X,\mathcal{B},\mu\right)$
be a standard probability space and $T:X\to X$ a measurable and invertible
transformation such that $\mu\circ T$ and $\mu$ have the same collection
of null sets ($\mu$ and $\mu\circ T$ are equivalent measures). Let
$\hat{T}:L^{1}\left(X,\mathcal{B},\mu\right)\to L^{1}\left(X,\mathcal{B},\mu\right)$
be the dual operator of $T$ defined by 
\[
\int_{X}f\cdot g\circ Td\mu=\int_{X}\left(\hat{T}f\right)gd\mu,
\]
 for all $g\in L^{\infty}\left(X,\mathcal{B},\mu\right)$ and $f\in L^{1}\left(X,\mathcal{B},\mu\right)$
. In our case, as $T$ is invertible, for all $n\in\mathbb{Z}$, 
\[
\hat{T}^{n}\left(f\right)(x)=\frac{d\left(\mu\circ T^{-n}\right)}{d\mu}(x)f\circ T^{-n}(x).
\]
A set $W\in\mathcal{B}$ is wandering if $\left\{ T^{n}W\right\} _{n\in\mathbb{Z}}$
are pairwise disjoint. The measurable union of all wandering sets,
denoted by $\mathfrak{D}(T)$, is called the dissipative part of $T$.
Its complement $\mathfrak{C}(T)=X\backslash\mathfrak{D}\left(T\right)$
is called the conservative part of $T$. The decomposition $X=\mathfrak{D}(T)\uplus\mathfrak{C}(T)$
is called the Hopf decomposition of $T$. The map $T$ is conservative
if there exists no wandering set $W$ of positive $\mu$ measure,
or equivalently $\mathfrak{C}(T)=X{\rm mod\mu}$. An equivalent definition
is that $T$ satisfies the conclusion of the Poincare recurrence theorem,
in the sense that for all $A\in\mathcal{B}$ with $\mu(A)>0$, for
almost every $x\in A$
\[
\sum_{k=1}^{\infty}1_{A}\circ T^{k}(x)=\infty.
\]
The conservative part, modulo a null set, is equal to
\[
\mathfrak{C}\left(T\right)=\left\{ x\in X:\ \sum_{k=1}^{\infty}\frac{d\left(\mu\circ T^{-n}\right)}{d\mu}(x)=\infty\right\} {\rm mod}\mu
\]
and $T$ is conservative if and only if
\[
\sum_{k=1}^{\infty}\frac{d\left(\mu\circ T^{-n}\right)}{d\mu}(x)=\infty,\ \mu-a.e.
\]
To shorten notation we will write ${\bf 1}$ for the constant function
${\bf 1}(x)=1$ and 
\[
\hat{T}^{n}{\bf 1(x)=}\frac{d\left(\mu\circ T^{-n}\right)}{d\mu}(x).
\]
Finally, by the Hurewicz ergodic theorem for all $A\in\mathcal{B}$
with $\mu\left(A\right)>0$, for $\mu$ almost every $x\in X$,
\[
\frac{\sum_{k=0}^{n-1}\hat{T}^{n}1_{A}(x)}{\sum_{k=1}^{\infty}\hat{T}^{n}{\bf 1}(x)}=\frac{\sum_{k=0}^{n-1}\hat{T}^{n}1_{A}(x)}{\sum_{k=1}^{\infty}\left(T^{-n}\right)'(x)}\xrightarrow[n\to\infty]{}h\left(1_{A},{\bf 1}\right)(x)
\]
where $h=h\left(A\right)\in L^{1}\left(X,\mathcal{B},\mu\right)$
satisfies:
\begin{itemize}
\item $h\circ T=h$ and $h\geq0$.
\item $\int_{X}h\psi d\mu=\int_{X}\psi1_{A}d\mu$ for all $\psi\in L^{\infty}\left(X,\mu\right)$
satisfying $\psi\circ T=\psi$. Consequently, the set $\left\{ x\in X:\ h(x)>0\right\} $
is of positive $\mu$ measure. 
\end{itemize}
The two bullets simply say that $h=\mathbb{E}\left(\left.1_{A}\right|\mathcal{I}\right)$,
where $\mathcal{I}$ is the $\sigma$-algebra of $T$ invariant sets.
One way of proving the Hurewicz ergodic theorem goes through the following
special case of the maximal inequality. Write $\hat{T}_{n}(f)=\sum_{k=0}^{n-1}\hat{T}^{k}f$. 
\begin{thm*}
Let $\left(X,\mathcal{B},\mu\right)$ be a conservative non-singular
transformation. Then for all $f\in L^{1}\left(X,\mathcal{B},\mu\right)$
and $t>0$, 
\[
\mu\left(x\in X:\ \sup_{n\in\mathbb{N}}\left|\frac{\hat{T}_{n}(f)}{\hat{T}_{n}\left({\bf 1}\right)}\right|>t\right)\leq\frac{\left\Vert f\right\Vert _{1}}{t}.
\]
\end{thm*}
The proof of Theorem \ref{thm: ergodicity of Bernoulli shifts} is
done by showing that for all $A\in\mathcal{B}$ with $\mu(A)>0$,
\[
\lim_{n\to\infty}\hat{T}_{n}\left(1_{A}\right)=\infty,\ \ \mu-a.e.
\]
This is equivalent to ergodicity by \cite[Proposition 1.3.2.]{Aar97InfErg}
and the fact that for a random variable $G:X\to[0,\infty]$, if for
all $A\in\mathcal{B}$ with $\mu\left(A\right)>0$, 
\[
\int_{A}Gd\mu=\infty,
\]
 then $G=\infty$ $\mu$ almost surely.

\subsection{Proof of Theorem \ref{thm: nonsingular erg}}

In this section, $F=\{1,...,N\}$, $X\subset F^{\mathbb{Z}}$ is a
subshift, $\mu$ is a probability measure supported on $X$ and $T$
denotes the shift on $F^{\mathbb{Z}}$. We assume that the measurable
equivalence relation
\[
\mathcal{T}=\left\{ \left(x,y\right)\in X\times X:\ \exists n\in\mathbb{N},\ x|_{\mathbb{Z}\backslash[-n,n]}=y|_{\mathbb{Z}\backslash[-n,n]}\right\} 
\]
is ergodic\footnote{Recall that this means that for all $A\in\mathcal{B}$, $\mu\left(\mathcal{T}\left(A\right)\right)=0$
or $\mu\left(X\backslash\mathcal{T}(A)\right)=0$}. In this section, $\mathcal{A}\subset\mathcal{B}$ denotes the collection
of finite union of cylinder sets in $\mathcal{B}$ where a cylinder
set is denoted by 
\[
\left[b\right]_{k}^{l}=\left\{ x\in F^{\mathbb{Z}}:\ \forall i\in[k,l]\cap\mathbb{Z},\ x_{i}=b_{i}\right\} ,
\]
where $b\in F^{\mathbb{Z}}$, $k,l\in\mathbb{Z}$.
\begin{prop}
Let $\left(X,\mathcal{B},\mu,T\right)$ be a conservative, non-singular
subshift which is double tail trivial. If there exists $L:\mathcal{T}\to\left(0,\infty\right)$
such that for $\mu\times\mu$ almost all $(x,y)\in\mathcal{T}$, for
all $n\in\mathbb{N}$, 
\begin{equation}
L(x,y)^{-1}\frac{d\left(\mu\circ T^{-n}\right)}{d\mu}(y)\leq\frac{d\left(\mu\circ T^{-n}\right)}{d\mu}(x)\leq L(x,y)\frac{d\left(\mu\circ T^{-n}\right)}{d\mu}(y)\label{eq: compare}
\end{equation}
then

(i) $T$ is either conservative or dissipative (either $\mathfrak{C}\left(T\right)=X\mod\mu$
or $\mathfrak{D}\left(T\right)=X\mod\mu$. 

(ii) If $T$ is conservative, then for all $A\in\mathcal{A}$, for
$\mu\times\mu$ almost all $(x,y)\in\mathcal{T}$, 
\[
\hat{T}_{n}\left(1_{A}\right)(x)\lesssim L(x,y)\hat{T}_{n}\left(1_{A}\right)(y).
\]

(iii) If $T$ is conservative, then for all $A\in\mathcal{A},$
\[
\lim_{n\to\infty}\hat{T}_{n}\left(1_{A}\right)(x)=\infty,\ \ \mu-a.e.
\]
\end{prop}
\begin{proof}
Write $\tilde{\mathcal{T}}\subset\mathcal{T}$ for the collection
of points on which (\ref{eq: compare}) holds and recall in what follows
that $\left(\mu\times\mu\right)\left(\mathcal{T}\backslash\mathcal{\tilde{T}}\right)=0$.
For all $\left(x,y\right)\in\tilde{T}$, we have for all $n\in\mathbb{N}$,
\[
\hat{T}_{n}\left({\bf 1}\right)(x)=\sum_{k=1}^{n}\frac{d\left(\mu\circ T^{-k}\right)}{d\mu}(x)=L\left(x,y\right)^{\pm1}\hat{T}_{n}\left({\bf 1}\right)(y).
\]
Consequently the set 
\[
\mathfrak{C}(T)=\left\{ x\in X:\ \lim_{n\to\infty}\hat{T}_{n}\left({\bf 1}\right)(x)=\infty\right\} 
\]
is $\mathcal{T}$ invariant in the sense that $\mathcal{T}\left(\mathfrak{C}\left(T\right)\right)=\mathfrak{C}\left(T\right)$
modulo $\mu$-null sets. By ergodicity of $\mathcal{T}$ either $\mathfrak{C}\left(T\right)=X\mod\mu$
or $\mathfrak{D}\left(T\right)=X\backslash\mathfrak{C}\left(T\right)=X$mod$\mu$,
showing part (i). 

\textbf{Proof of (ii) and (iii)}. Let $A$ be a cylinder set and $x,y\in X$
such that $\left(x,y\right)\in\mathcal{T}$. Suppose that (\ref{eq: compare})
holds and $\lim_{n\to\infty}\hat{T}_{n}\left(1_{A}\right)(x)=\infty$.
In this case for all $n\in\mathbb{N}$,
\[
\hat{T}_{n}\left(1_{A}\right)(x)=\sum_{k=0}^{n-1}\frac{d\left(\mu\circ T^{-n}\right)}{d\mu}(x)1_{A}\circ T^{-k}(x)\leq L(x,y)\sum_{k=0}^{n-1}\frac{d\left(\mu\circ T^{-n}\right)}{d\mu}(y)1_{A}\circ T^{-k}(x)
\]
This shows, as the left hand side tends to infinity as $n\to\infty$,
that 
\begin{equation}
\sum_{k=0}^{n-1}\frac{d\left(\mu\circ T^{-n}\right)}{d\mu}(y)1_{A}\circ T^{-k}(x)\xrightarrow[n\to\infty]{}\infty.\label{eq: infinity-1}
\end{equation}
As $\left(x,y\right)\in\mathcal{T}$ and $A$ is a cylinder set, there
exists $n_{0}=n_{0}(x,y,A)$ such that if $n>n_{0}$, then $x\in T^{n}A$
if and only if $y\in T^{n}A$. This together with (\ref{eq: infinity-1})
imply that as $n\to\infty$, 
\[
\sum_{k=0}^{n-1}\frac{d\left(\mu\circ T^{-n}\right)}{d\mu}(y)\left(1_{A}\circ T^{-k}(x)\right)\sim\sum_{k=0}^{n-1}\sum_{k=0}^{n-1}\frac{d\left(\mu\circ T^{-n}\right)}{d\mu}(y)\left(1_{A}\circ T^{-k}(y)\right)=\hat{T}_{n}\left(1_{A}\right)(y).
\]

We have shown that if $x,y\in X$, $\left(x,y\right)\in\tilde{\mathcal{T}}$
and $\lim_{n\to\infty}\hat{T}_{n}\left(1_{A}\right)(x)=\infty$, then
\begin{equation}
\hat{T}_{n}\left(1_{A}\right)(x)\lesssim L(x,y)\hat{T}_{n}\left(1_{A}\right)(y)\label{eq: oof}
\end{equation}
as $n\to\infty$, thus
\[
\lim_{n\to\infty}\hat{T}_{n}\left(1_{A}\right)(y)=\infty.
\]
This implies that the set 
\[
{\bf \tilde{A}}=\left\{ x\in X:\ \lim_{n\to\infty}\hat{T}_{n}\left(1_{A}\right)(x)=\infty\right\} 
\]
is $\mathcal{T}$ invariant. 

As $T$ is conservative, by the Hurewicz ergodic theorem the set 
\[
\left\{ x\in X:\ h\left(A\right)(x)>0\right\} \subset{\bf \tilde{A}}
\]
is of positive measure. Consequently, by ergodicity of $\mathcal{T}$,
$\mu\left(X\backslash\tilde{{\bf A}}\right)=0$ proving part (iii).
Part (ii) follows from part (iii) and (\ref{eq: oof}). 
\end{proof}
\begin{proof}[Proof of Theorem \ref{thm: ergodicity of Bernoulli shifts}]

Assume in the contrapositive that there exists $B,D\in\mathcal{B}$
of positive $\mu$ measure such that for all $x\in D$, 
\[
\sum_{n=0}^{\infty}\hat{T}^{n}\left(1_{B}\right)(x)<\infty.
\]
By the ratio ergodic theorem, there exists $\epsilon>0$ for which
the set 
\[
C=\left\{ x\in X:\ \lim_{n\to\infty}\frac{\hat{T}_{n}\left(1_{B}\right)(x)}{\hat{T}_{n}\left({\bf 1}\right)(x)}=h\left(B\right)>2\epsilon\right\} 
\]
satisfies $\mu\left(C\right)>0$. Secondly, there exists $A_{n}\in\mathcal{A}$
such that 
\[
\left\Vert 1_{A_{n}}-1_{B}\right\Vert _{1}=\mu\left(A_{n}\triangle B\right)\leq\frac{1}{n^{2}}.
\]

By the maximal inequality, 
\[
\mu\left(x\in X:\ \sup_{n\in\mathbb{N}}\left|\frac{\hat{T}_{n}\left(1_{B}-1_{A_{k}}\right)(x)}{\hat{T}_{n}\left({\bf 1}\right)(x)}\right|>\epsilon\right)\leq\frac{1}{n^{2}\epsilon}.
\]
As the right hand side is summable, it follows from the Borel-Cantelli
Lemma that the set 
\[
A=\left\{ x\in X:\ \exists K\in\mathbb{N},\forall k>K,\ \sup_{n\in\mathbb{N}}\left|\frac{\hat{T}_{n}\left(1_{B}-1_{A_{k}}\right)(x)}{\hat{T}_{n}\left({\bf 1}\right)(x)}\right|<\epsilon\right\} 
\]
is of full $\mu$-measure. Consequently, the set $E=\bigcup_{K\in\mathbb{N}}E_{K}$
\[
E_{K}=\left\{ x\in X:\ \forall k>K,\ \liminf_{n\to\infty}\frac{\hat{T}_{n}\left(1_{A_{k}}\right)(x)}{\hat{T}_{n}\left({\bf 1}\right)(x)}>\epsilon\right\} ,
\]
satisfies 
\[
C\cap A\subset C\cap E,
\]
whence 
\[
\mu\left(E\right)\geq\mu(C)>0.
\]
To see the set inclusion, notice that for all $x\in C$, 
\[
\lim_{n\to\infty}\frac{\hat{T}_{n}\left(1_{B}\right)(x)}{\hat{T}_{n}\left({\bf 1}\right)(x)}>2\epsilon.
\]
Now, if $x\in A\cap C$, there exists $K$ such that for all $k>K$,
\[
\sup_{n\in\mathbb{N}}\left|\frac{\hat{T}_{n}\left(1_{B}-1_{A_{k}}\right)(x)}{\hat{T}_{n}\left({\bf 1}\right)(x)}\right|<\epsilon.
\]
Thus, for all $k>K$, 
\[
\liminf_{n\to\infty}\frac{\hat{T}_{n}\left(1_{A_{k}}\right)(x)}{\hat{T}_{n}\left({\bf 1}\right)(x)}\geq\lim_{n\to\infty}\frac{\hat{T}_{n}\left(1_{B}\right)(x)}{\hat{T}_{n}\left({\bf 1}\right)(x)}-\sup_{n\in\mathbb{N}}\left|\frac{\hat{T}_{n}\left(1_{B}-1_{A_{k}}\right)(x)}{\hat{T}_{n}\left({\bf 1}\right)(x)}\right|>\epsilon.
\]
and therefore $x\in E$. As $\mu\left(E\right)>0$ and for all $K\in\mathbb{N}$,
$E_{K}\subset E_{K+1}$, it follows that for all large $K$, $\mu\left(E_{K}\right)>0$.
By ergodicity of $\mathcal{T}$, for all large $K$, 
\begin{equation}
\mu\left(\mathcal{T}\left(E_{K}\right)\cap D\right)=\mu\left(D\right).\label{eq: set equal-1}
\end{equation}
From now on we assume that $K$ is large enough so that (\ref{eq: set equal-1})
holds. For almost every $y\in D$, there exists $x\in E$ with $x\sim y$.
As $A_{k}$ is a finite union of cylinder sets, by Proposition \ref{prop: main}
part (ii), for all $n\in\mathbb{N}$ and $k\in\mathbb{N},$ for almost
every $y\in X$, 
\[
\frac{\hat{T}_{n}\left(1_{A_{k}}\right)(y)}{\hat{T}_{n}\left({\bf 1}\right)(y)}\gtrsim\left(\frac{1}{L(x,y)}\right)^{2}\frac{\hat{T}_{n}\left(1_{A_{k}}\right)(x)}{\hat{T}_{n}\left({\bf 1}\right)(x)},
\]
thus 
\[
\liminf_{n\to\infty}\frac{\hat{T}_{n}\left(1_{A_{k}}\right)(y)}{\hat{T}_{n}\left({\bf 1}\right)(y)}\geq\left(\frac{1}{L(x,y)}\right)^{2}\liminf_{n\to\infty}\frac{\hat{T}_{n}\left(1_{A_{k}}\right)(x)}{\hat{T}_{n}\left({\bf 1}\right)(x)}.
\]
By the definition of $E$, we have shown that for almost every $y\in D$,
there exists $L(y)\in\mathbb{N}$ such that for all $k>K$, 
\[
\liminf_{n\to\infty}\frac{\hat{T}_{n}\left(1_{A_{k}}\right)(y)}{\hat{T}_{n}\left({\bf 1}\right)(y)}\geq\frac{\epsilon}{L(y)}.
\]
By this, there exists $D'\subset D$ with $\mu\left(D'\right)>0$
and ${\bf L},\mathbf{K}\in\mathbb{N}$ , such that for all $y\in D'$
and $k>\mathbf{K}$,
\[
\liminf_{n\to\infty}\frac{\hat{T}_{n}\left(1_{A_{k}}\right)(y)}{\hat{T}_{n}\left({\bf 1}\right)(y)}>\frac{\epsilon}{{\bf L}}.
\]
For $y\in D'$, for all $n\in\mathbb{N}$ and $k>{\bf K}$, 
\begin{align*}
\frac{\epsilon}{{\bf L}}\hat{T}_{n}\left({\bf 1}\right)(y)-\left|\hat{T}_{n}\left(1_{B}-1_{A_{k}}\right)\right|(y) & \lesssim\hat{T}_{n}\left(1_{A_{k}}\right)(y)-\left|\hat{T}_{n}\left(1_{B}-1_{A_{k}}\right)\right|(y)\\
 & \leq\hat{T}_{n}\left(1_{B}\right)(y)\\
 & \leq\sum_{j=0}^{\infty}\hat{T}^{j}\left(1_{B}\right)(y)<\infty,\ \text{as}\ y\in D.
\end{align*}
This shows that for all $y\in D'$ and $k>{\bf K}$
\[
\liminf_{n\to\infty}\frac{\left|\hat{T}_{n}\left(1_{B}-1_{A_{k}}\right)\right|(y)}{\hat{T}_{n}\left({\bf 1}\right)(y)}\geq\frac{\epsilon}{{\bf L}}.
\]
This is a contradiction to $\mu\left(D'\right)>0$, since for all
$k>\mathbf{K}$
\[
\mu\left(D'\right)\leq\mu\left(x\in X:\ \liminf_{n\to\infty}\frac{\left|\hat{T}_{n}\left(1_{B}-1_{A_{k}}\right)\right|(y)}{\hat{T}_{n}\left({\bf 1}\right)(y)}\geq\frac{\epsilon}{{\bf L}}\right)
\]

\[
\leq\mu\left(x\in X:\ \sup_{n\in\mathbb{N}}\frac{\left|\hat{T}_{n}\left(1_{B}-1_{A_{k}}\right)\right|(y)}{\hat{T}_{n}\left({\bf 1}\right)(y)}\geq\frac{\epsilon}{{\bf L}}\right)\leq\frac{{\bf L}}{\epsilon k^{2}}\xrightarrow[k\to\infty]{}0.
\]
The result follows. 

\end{proof}

\section{Examples\label{sec:Examples}}

\subsection{The model of non-singular Bernoulli shifts}

Let $N\in\mathbb{N}$, $X=\left\{ 1,..,N\right\} ^{\mathbb{Z}}$ and
$\mathcal{B}=\mathcal{B}_{X}$ the Borel $\sigma-$algebra of $X$
which is generated by the collection of cylinder sets 
\[
\left\{ [b]_{l}^{k}:\ k.l\in\mathbb{Z},\ b\in X\right\} .
\]
 Given a sequence $\left(\mu_{k}\right)_{k\in\mathbb{Z}}$ of probability
measures on $\left\{ 1,..,N\right\} $, the product measure $\mu=\prod_{k\in\mathbb{Z}}\mu_{k}$
is the measure on $X$ defined by
\[
\mu\left(\left[b\right]_{l}^{k}\right)=\prod_{j=l}^{k}\mu_{j}\left(b_{j}\right).
\]
In other words, $\mu$ is the distribution of an independent sequence
of random variables $\left(X_{n}\right)_{n\in\mathbb{Z}}$, where
for all $n$, $X_{n}$ is distributed according to $\mu_{n}$. A non-singular
Bernoulli shift on $N$ symbols is the quadruple $\left(X,\mathcal{B},\mu,T\right)$
where $T$ is the shift map on $X$, and $\mu\circ T$ is equivalent
to $\mu$ (i.e. the shift is a nonsingular transformation). 

As $\mu$ and $\mu\circ T$ are product measures, the following is
a direct consequence of Kakutani's theorem on equivalence of product
measures. From now on, when the measure $\mu$ is fixed, we denote
by $\left(T^{n}\right)'=\frac{d\left(\mu\circ T^{n}\right)}{d\mu}$. 
\begin{prop}
$\left(\left\{ 1,...,N\right\} ^{\mathbb{Z}},\mathcal{B},\mu=\prod_{k=-\infty}^{\infty}\mu_{k},T\right)$
is non-singular if and only if (\ref{eq: Kakutani}) holds. In that
case there exists $X'\subset\left\{ 1,...,N\right\} ^{\mathbb{Z}}$
with $\mu\left(X'\right)=1$ such that for all $x\in X'$, 
\[
\left(T^{n}\right)'(x)=\prod_{k=-\infty}^{\infty}\frac{P_{k-n}\left(x_{k}\right)}{P_{k}\left(x_{k}\right)}.
\]
\end{prop}
One should note that if for some $k\in\mathbb{Z}$ there exists $j\in\{1,..,N\}$
with $\mu_{k}\left(\left\{ j\right\} \right)=1$, then the shift is
non-singular if and only if $\mu=\prod_{k\in\mathbb{Z}}\delta_{\{j\}}$,
which is supported on a single point in $X$. As this case is not
interesting, we always assume that for all $k\in\mathbb{Z},$
\begin{equation}
M_{k}=\max_{j\in\left\{ 1,..,N\right\} }\mu_{k}\left(\{j\}\right)<1.\label{eq: maximum prob}
\end{equation}
 In addition, by a similar argument, if $\mu\circ T\sim\mu$ and for
some $k\in\mathbb{Z}$, there exists $j\in\{1,..,N\}$ with $\mu_{k}\left(\left\{ j\right\} \right)=0$
then for all $m\in\mathbb{Z}$, 
\[
\mu_{m}(\{j\})=0
\]
and we can reduce $\left(\left\{ 1,...,N\right\} ^{\mathbb{Z}},\mathcal{B},\mu=\prod_{k=-\infty}^{\infty}\mu_{k},T\right)$
to $\left(\left\{ 1,..,N'\right\} ^{\mathbb{Z}},\mathcal{B},\mu'=\prod_{k=-\infty}^{\infty}\mu'_{k},T\right)$
with $N'<N$. As our statement holds for all $N\in\mathbb{N}$, we
will henceforth assume that for all $k\in\mathbb{Z}$, 
\begin{equation}
m_{k}=\min_{j\in\left\{ 1,..,N\right\} }\mu_{k}\left(\{j\}\right)>0.\label{eq: min prob}
\end{equation}

We say that $x,y\in X$ are double tail equivalent, denoted by $x\sim y$,
if there exists $N=N(x,y)$ such that for all $|n|>N$, 
\[
x_{n}=y_{n}.
\]

\begin{lem}
\label{claim: double tail equiv}Let $\left(\left\{ 1,...,N\right\} ^{\mathbb{Z}},\mathcal{B},\mu=\prod_{k=-\infty}^{\infty}\mu_{k},T\right)$
be a non-singular Bernoulli shift satisfying (\ref{eq: maximum prob})
and (\ref{eq: min prob}). There exists $X'\in\mathcal{B}$ with $\mu\left(X'\right)=1$
such that if $x,y\in X'$ and $x\sim y$, then for all $n\in\mathbb{Z}$
\[
\prod_{k=-N(x,y)}^{N(x,y)}\left(\frac{m_{k}m_{k-n}}{M_{k}M_{k-n}}\right)\leq\frac{\left(T^{n}\right)'(x)}{\left(T^{n}\right)'(y)}\leq\prod_{k=-N(x,y)}^{N(x,y)}\left(\frac{M_{k}M_{k-n}}{m_{k}m_{k-n}}\right).
\]
 In particular, if condition (\ref{eq: condition on B-shifts}) holds
then for all $x,y\in X'$ with $x\sim y$
\[
L^{-4N(x,y)}\leq\inf_{n\in\mathbb{Z}}\frac{\left(T^{n}\right)'(x)}{\left(T^{n}\right)'(y)}\leq\sup_{n\in\mathbb{Z}}\frac{\left(T^{n}\right)'(x)}{\left(T^{n}\right)'(y)}\leq L^{4N(x,y)}.
\]
\end{lem}
The following is a double tail $\{0,1\}$-law. It is certainly not
new, a proof is presented here for the sake of completeness. 
\begin{lem}
\label{prop: main} The double tail relation of a conservative, non-singular
Bernoulli shift $\left(\left\{ 1,..,N\right\} ^{\mathbb{Z}},\mathcal{B},\mu=\prod_{k=-\infty}^{\infty}\mu_{k},T\right)$
is ergodic. 
\end{lem}
\begin{proof}
It is enough to show that for every set $B\in\mathcal{B}$ with $\mu(B)>0$,
the set 
\[
\mathcal{T}(B)=\left\{ y\in X:\ \exists x\in B,\ \left(x,y\right)\in\mathcal{T}\right\} 
\]
is of full $\mu$ measure. To see this, let $\epsilon>0$. As the
collection of cylinder sets is dense in $\mathcal{B}$, there exists
$C=[c]_{r}^{m}$ such that 
\[
\mu\left(C\cap B\right)\geq(1-\epsilon)\mu(C).
\]
 For $x\in C\cap B$ and $[z]_{r}^{m}=Z$ another cylinder set, the
point $y=y(x,Z)\in X$ with 
\[
y_{i}=\begin{cases}
x_{i}, & i\notin[r,m],\\
z_{i} & i\in[r,m]
\end{cases}
\]
satisfies $\left(x,y\right)\in\mathcal{T}$. This shows that for all
$z\in X$, 
\[
\mu\left([z]_{r}^{m}\cap\mathcal{T}(B)\right)\geq(1-\epsilon)\mu\left([z]_{r}^{m}\right)
\]
whence as two different $r,m$ cylinders are disjoint we see that
\[
\mu\left(\mathcal{\mathcal{T}}(B)\right)=\sum_{Z=[z]_{r}^{m}}\mu\left(Z\cap\mathcal{T}(B)\right)\geq1-\epsilon.
\]

Since $\epsilon$ is arbitrary, we see that $\mu\left(\mathcal{T}\left(B\right)\right)=1$. 
\end{proof}
\begin{proof}[Proof of Theorem \ref{thm: ergodicity of Bernoulli shifts}]

Let $\left(\left\{ 1,...,N\right\} ^{\mathbb{Z}},\mathcal{B},\mu,T\right)$
be a conservative nonsingular Bernoulli shift which satisfies condition
(\ref{eq: condition on B-shifts}). By Lemma \ref{prop: main} its
double tail is ergodic. By Lemma \ref{claim: double tail equiv} it
satisfies the conditions of Theorem \ref{thm: nonsingular erg}, hence
it is ergodic. 

\end{proof}

\subsection{Inhomogeneous Markov shifts supported on topologically mixing subshifts
of finite type}

Let $S$ be a finite space and $A=\left(A(s,t)\right)_{s,t\in S}$
be an $S\times S$ $\{0,1\}$-valued matrix. The shift invariant set
\[
\Sigma_{A}=\left\{ x\in S^{\mathbb{Z}}:\ \forall i\in\mathbb{Z},\ A\left(x_{i},x_{i+1}\right)=1\right\} 
\]
is a subshift of finite type (SFT). It is topologically mixing iff
there exists $n\in\mathbb{N}$ such that $A^{n}$ has all entries
positive (i.e. $A$ is primitive). 

An $S$ valued inhomogeneous Markov shift consists of a sequence of
$S\times S$ stochastic matrices $\left(P_{n}\right)_{n\in\mathbb{Z}}$
and a sequence of probability distributions $\left(\pi_{n}\right)_{n\in\mathbb{Z}}$
regarded as row vectors satisfying for all $j\in\mathbb{Z}$, 
\[
\pi_{j}P_{j}=\pi_{j-1}.
\]
 With this condition the measure $\mu$ defined on the collection
of cylinder sets by 
\[
\mu\left(\left[b\right]_{k}^{l}\right)=\pi_{k}\left(b_{k}\right)\prod_{j=k}^{l-1}P_{j}\left(b_{j},b_{j+1}\right)
\]
has a unique extension to a measure $\mu=\mu\left(\left(P_{n}\right)_{n\in\mathbb{Z}},\left(\pi_{n}\right)_{n\in\mathbb{Z}}\right)$
on all $\mathcal{B}_{S^{\mathbb{Z}}}$. Writing $X_{i}:S^{\mathbb{Z}}\to S$
for the projection to the $i$-th coordinate, the sequences $P_{n}$
and $\pi_{n}$ have the following interpretation, which is the standard
definition of an inhomogeneous Markov chain:
\begin{align*}
\pi_{n}(s) & =\mu\left(X_{n}=s\right)\\
P_{n}(s,t) & =\mu\left(\left.X_{n+1}=t\right|X_{n}=s\right)=\mu\left(\left.X_{n+1}=t\right|X_{n}=s,X_{n-1},...\right).
\end{align*}

In this section we assume that the measure $\mu$ is fully supported
on $\Sigma_{A}$, in the sense that for all $n\in\mathbb{Z}$, 
\[
{\rm supp}\left(P_{n}\right)=\left\{ (s,t):\ P_{n}(s,t)>0\right\} =\left\{ (s,t):\ A(s,t)=1\right\} ={\rm supp}\left(A\right).
\]

\begin{thm}
\label{thm: TMS}Let $\Sigma_{A}\subset S^{\mathbb{Z}}$ be a topologically
mixing Markov shift. Assume that $\mu=\mu\left(\left(P_{n}\right)_{n\in\mathbb{Z}},\left(\pi_{n}\right)_{n\in\mathbb{Z}}\right)$
is fully supported on $\Sigma_{A}$ and 
\begin{equation}
\sup_{n\in\mathbb{Z}}\sup_{s\in S}\left(\frac{P_{n}\left(s,t\right)}{P_{n}\left(s,t'\right)}:\ t,t'\in S,\ P_{n}\left(s,t'\right)>0\right)=L<\infty.\label{eq: comm}
\end{equation}
If the shift $\left(\Sigma_{A},\mathcal{B}_{\Sigma_{A}},\mu,T\right)$
is nonsingular and conservative, then it is ergodic.
\end{thm}
A criterion for non-singularity of the shift can be obtained in the
following way using the method of \cite{KabLipShi1977}; full details
and proofs of these statements are also in \cite{Shiryaev_prob}.Write
\[
\mathcal{F}_{n}=\left\{ [b]_{-n}^{n}:\ b\in\Sigma_{A}\right\} 
\]
for the collection of symmetric cylinder sets and for a measure $\nu$
on $\mathcal{B}_{\Sigma_{A}}$, write $\nu_{n}$ for the measure $\nu$
restricted to $\mathcal{F}_{n}$. For $\mu=\mu\left(\left(P_{n}\right),\left(\pi_{n}\right)\right)$
an inhomogeneous Markov measure, $\mu\circ T$ is the Markov measure
with transition matrices $Q_{n}=P_{n-1}$ and $\tilde{\pi}_{n}=\pi_{n-1}$.
A necessary condition for $T$ to be nonsingular is that $\left(\mu\circ T\right)_{n}$
and $\mu_{n}$ are absolutely continuous for all $n\in\mathbb{N}$,
which is referred to in \cite{Shiryaev_prob} as local absolute continuity.
This amounts to the condition that for all $n\in\mathbb{Z}$,
\[
\mu\left([b]_{-n}^{n}\right)>0\ \ \Leftrightarrow\ \mu\circ T\left([b]_{-n}^{n}\right)>0.
\]
In that case one defines 
\[
Z_{n}(x)=\frac{d\left(\mu\circ T\right)_{n}}{d\mu_{n}}(x)=\frac{\pi_{-n-1}\left(x_{-n}\right)}{\pi_{-n}\left(x_{-n}\right)}\cdot\prod_{j=-n}^{n-1}\frac{P_{j-1}\left(x_{j},x_{j+1}\right)}{P_{j}\left(x_{j},x_{j+1}\right)}.
\]
The sequence $\left\{ Z_{n}\right\} _{n=1}^{\infty}$ is a martingale
with respect to the filtration $\mathcal{F}_{n}$ and $\mathcal{F}_{n}\uparrow\mathcal{B}_{\Sigma_{A}}$.
Thus $Z_{n}$ converges almost surely to a $[0,\infty]$-valued random
variable. It then follows that $\mu$ and $\mu\circ T$ are equivalent
measures if and only if $Z_{n}$ converges in $L^{1}$, which is equivalent
to uniform integrability of $\left\{ Z_{n}\right\} _{n=1}^{\infty}$.
For a streamlined discussion of necessary and sufficient conditions
see \cite{Shiryaev_prob}. We will only make use of the form of the
Radon-Nykodym derivatives which is summarized in the following lemma.
\begin{lem}
Let $\mu=\mu\left(\left(P_{n}\right)_{n\in\mathbb{Z}},\left(\pi_{n}\right)_{n\in\mathbb{Z}}\right)$
be an inhomogeneous Markov chain with state space $S$. If $\mu\circ T\sim\mu$,
then there exists $X'\subset S^{\mathbb{Z}}$ with $\mu\left(X'\right)=1$
such that for all $x\in X'$ and $N\in\mathbb{Z}$,
\[
\frac{d\left(\mu\circ T^{N}\right)}{d\mu}\left(x\right)=\lim_{n\to\infty}\left(\frac{\pi_{-n-N}\left(x_{-n}\right)}{\pi_{-n}\left(x_{-n}\right)}\cdot\prod_{j=-n}^{n-1}\frac{P_{j-N}\left(x_{j},x_{j+1}\right)}{P_{j}\left(x_{j},x_{j+1}\right)}\right).
\]
\end{lem}
\begin{prop}
\label{prop: derivative condition}Under the assumptions of Theorem
\ref{thm: TMS}, there exists $L(x,y):\mathcal{T}\to(0,\infty)$ such
that for all $(x,y)\in\left(X'\times X'\right)\cap\mathcal{T}$, 
\[
L(x,y)^{-1}\frac{d\mu\circ T^{N}}{d\mu}\left(y\right)\leq\frac{d\mu\circ T^{N}}{d\mu}\left(x\right)\leq L(x,y)\frac{d\mu\circ T^{N}}{d\mu}\left(y\right).
\]
\end{prop}
\begin{proof}
For two $S\times S$ matrices $A,B$, we write $A\leq B$ if for all
$(s,t)\in S\times S$, 
\[
A(s,t)\leq B(s,t).
\]
Firstly, it follows from (\ref{eq: comm}) that if $P_{n}\left(s,t\right)>0$,
then 
\begin{equation}
P_{n}\left(s,t\right)\geq\frac{1}{L^{|S|}}.\label{eq: yeeess}
\end{equation}
Indeed, there are at most $|S|$ elements $t\in S$ such that $P_{n}(s,t)>0$.
Organizing them as an increasing sequence, we get that for all $s\in S$,
\[
\frac{\max\left(P_{n}(s,t):\ t\in S\right)}{\min\left(P_{n}(s,t):\ t\in S,P_{n}(s,t)>0\right)}\leq L^{|S|}.
\]
This implies (\ref{eq: yeeess}). Secondly, as $\mu$ is fully supported,
this implies that for all $m\leq n$, 
\begin{equation}
P^{(m,n)}:=P_{m}P_{m+1}\cdots P_{n}\geq L^{-|S|(n-m+1)}A^{n-m+1}.\label{eq: important}
\end{equation}
Thirdly, as $\mu$ is fully supported and $P_{n}$ are stochastic
matrices, it follows that for all $n\in\mathbb{Z}$ and $s,t\in S$,
\[
P_{n}(s,t)\leq A(s,t)
\]
and thus for all $m\leq n$,
\begin{equation}
P^{(m,n)}\leq A^{n-m+1}.\label{eq: slightly less}
\end{equation}
Finally, let $(x,y)\in\left(X'\times X'\right)\cap\mathcal{T}$, $N\in\mathbb{Z}$
and $n(x,y)\in\mathbb{N}$ such that for all $K\in\mathbb{Z}$ with
$|K|>n(x,y)$, 
\[
x_{K}=y_{K}.
\]
Then, for all $K$ such that $K-N>n(x,y)$, 
\[
\frac{\pi_{-K-N}\left(x_{-K}\right)}{\pi_{-K}\left(x_{-K}\right)}\cdot\prod_{j=-K}^{K-1}\frac{P_{j-N}\left(x_{j},x_{j+1}\right)}{P_{j}\left(x_{j},x_{j+1}\right)}=\left(\frac{\pi_{-K-N}\left(y_{-K}\right)}{\pi_{-K}\left(y_{-K}\right)}\cdot\prod_{j=-K}^{K-1}\frac{P_{j-N}\left(y_{j},y_{j+1}\right)}{P_{j}\left(y_{j},y_{j+1}\right)}\right)\cdot I(x,y),
\]
where 
\[
I(x,y)=\prod_{j=-n(x,y)}^{n(x,y)}\left(\frac{P_{j-N}\left(x_{j},x_{j+1}\right)}{P_{j-N}\left(y_{j},y_{j+1}\right)}\frac{P_{j}\left(y_{j},y_{j+1}\right)}{P_{j}\left(x_{j},x_{j+1}\right)}\right).
\]
As all elements in the product in $I(x,y)$ are strictly positive
(since $x,y\in\Sigma_{A}$), it follows from (\ref{eq: important})
and (\ref{eq: slightly less}) that 
\[
I(x,y)\leq L^{4|S|n(x,y)}=:L(x,y).
\]
We have shown that
\[
\frac{\pi_{-K-N}\left(x_{-K}\right)}{\pi_{-K}\left(x_{-K}\right)}\cdot\prod_{j=-K}^{K-1}\frac{P_{j-N}\left(x_{j},x_{j+1}\right)}{P_{j}\left(x_{j},x_{j+1}\right)}=\left(\frac{\pi_{-K-N}\left(y_{-K}\right)}{\pi_{-K}\left(y_{-K}\right)}\cdot\prod_{j=-K}^{K-1}\frac{P_{j-N}\left(y_{j},y_{j+1}\right)}{P_{j}\left(y_{j},y_{j+1}\right)}\right)\cdot L(x,y).
\]
Taking the limit as $K\to\infty$, we see that 
\[
\frac{d\left(\mu\circ T^{N}\right)}{d\mu}(x)\leq L(x,y)\frac{d\left(\mu\circ T^{N}\right)}{d\mu}(y).
\]
The proof is complete as the roles of $x$ and $y$ are symmetric
(thus the lower bound). 
\end{proof}
In order to prove Theorem \ref{thm: TMS} it remains to show that
the double tail $\mathcal{T}$ is trivial under the assumptions of
the Theorem. This is the following proposition. 
\begin{thm}
\label{thm: double tail TMS}Under the assumptions of Theorem \ref{thm: TMS}
the double tail relation $\mathcal{T}\subset\Sigma_{A}\times\Sigma_{A}$
is trivial. 
\end{thm}
\begin{proof}
Let $N\in\mathbb{N}$ such that $A^{N}>0$. Since for all $n\in\mathbb{Z}$,
if $P_{n}(s,t)>0$ then 
\[
P_{n}(s,t)>L^{-|S|}
\]
we see that for any path $s=s_{0},s_{1},...,s_{N}=t$ such that 
\[
\prod_{i=0}^{N-1}P_{n+i}\left(s_{i},s_{i+1}\right)>0
\]
we have 
\[
\prod_{i=0}^{N-1}P_{n+i}\left(s_{i},s_{i+1}\right)\geq L^{-|S|N}.
\]

As $A^{N}>0$ and $\mu$ is fully supported on $\Sigma_{A}$ it follows
that for all $s,t\in S$ there exists a path $s=s_{0},s_{1},...,s_{N}=t$
such that
\[
\prod_{i=0}^{N-1}P_{n+i}\left(s_{i},s_{i+1}\right)\geq L^{-|S|N}.
\]
 Let $B=[b]_{-n}^{n}$ and $C=[c]_{-n}^{n}$ be arbitrary symmetric
$n$-cylinders of positive $\mu$ measure. We claim that 
\begin{equation}
\mu\left(\mathcal{T}\left(B\right)\cap C\right)\geq|S|^{-1}L^{-2|S|N}\mu(C).\label{eq: muyimportante}
\end{equation}
In order to prove this, take $s\in S$ such that 
\[
\pi_{-n-N}\left(s\right)\geq|S|^{-1}.
\]
Such states exist as $\pi_{n},\pi_{-n}$ are probability distributions
on $S$. By the first part, there exist a path $s_{-n-N}=s,s_{-n-(N-1)},...,s_{-n}=b_{-n}$
such that 
\[
\prod_{i=-n+N}^{-n-1}P_{i}\left(s_{i},s_{i+1}\right)>L^{-|S|N}
\]
Similarly, there exists a path $s_{n}=b_{n},...,s_{n+N}=s$ such that
\[
\prod_{i=n}^{n+N-1}P_{i}\left(s_{i},s_{i+1}\right)>L^{-|S|N}.
\]
Defining a symmetric $n+N$ cylinder $B'=[b']_{-n+N}^{n+N}$ via 
\[
b'=\begin{cases}
b_{i} & i\in[-n,n]\\
s_{i}, & i\in[-n-N,n+N]\backslash[-n,n],
\end{cases}
\]
it follows that $B'\subset B$ and 
\begin{align*}
\frac{\mu\left(B'\right)}{\mu(B)} & =\frac{\pi_{-n-N}\left(s\right)}{\pi_{-n}\left(b_{-n}\right)}\left(\prod_{i=-n+N}^{-n-1}P_{i}\left(s_{i},s_{i+1}\right)\right)\left(\prod_{i=n}^{n+N-1}P_{i}\left(s_{i},s_{i+1}\right)\right)\\
 & \geq|S|^{-1}L^{-|S|N}.
\end{align*}
An identical argument constructs a symmetric $n+N$ cylinder $C'=[c']_{-n+N}^{n+N}$
with $C'\subset C$, 
\[
c_{-n-N},\ c_{n+N}=s
\]
and 
\[
\frac{\mu\left(C'\right)}{\mu(C)}\geq|S|^{-1}L^{-|S|N}.
\]
For every $x\in B'$, define $R(x)\in C'$ by 
\[
R(x)_{i}=\begin{cases}
x_{i}, & i\notin[-n-N,n+N]\\
c_{i}, & i\in[-n-N,n+N].
\end{cases}
\]
The map $R:B'\to C'$ is bijective and for all $x\in B'$, $\left(x,R(x)\right)\in\mathcal{T}$.
Thus 
\[
\mu\left(\mathcal{T}\left(B\right)\cap C\right)\geq\mu\left(C'\right)\geq|S|^{-1}L^{-|S|N}\mu(C),
\]
proving (\ref{eq: muyimportante}). Another feature of $R$ is that
for all $x\in B'$, 
\[
\frac{d\mu\circ R}{d\mu}(x)=\frac{\mu\left(C'\right)}{\mu\left(B'\right)}.
\]
As a consequence, if $A\subset B'$, then writing $\epsilon=|S|^{-1}L^{-|S|N}$,
\begin{align}
\mu\left(\mathcal{T}\left(A\right)\cap C\right) & \geq\mu\left(R(A)\right)\nonumber \\
 & =\frac{\mu\left(A\right)}{\mu\left(B'\right)}\mu\left(C'\right)\geq\epsilon\frac{\mu\left(A\right)}{\mu\left(B\right)}\mu\left(C\right).\label{eq: the real deal}
\end{align}

Now let $D\in\mathcal{B}_{\Sigma_{A}}$ be a tail invariant set. If
$\mu\left(D\right)>0$ and $\mu\left(\Sigma_{A}\backslash D\right)>0$,
then there exists $n\in\mathbb{N}$ and two cylinder sets $B=[b]_{-n}^{n},C=[c]_{-n}^{n}$
such that 
\[
\mu\left(D\cap B\right)\geq\left(1-\frac{\epsilon}{2}\right)\mu(B)\ \ \text{and }\mu\left(D\cap C\right)<\frac{\epsilon^{2}}{4}\mu\left(C\right).
\]
As 
\[
\mu\left(B'\right)\geq\epsilon\mu\left(B\right),
\]
it then follows that 
\[
\mu\left(D\cap B'\right)\geq\frac{\epsilon}{2}\mu\left(B\right).
\]
Consequently, by (\ref{eq: the real deal}), 
\begin{align*}
\mu\left(D\cap C\right) & =\mu\left(\mathcal{T}\left(D\right)\cap C\right)\\
 & \geq\mu\left(\mathcal{T}\left(D\cap B'\right)\cap C\right)\\
 & \geq\epsilon\frac{\mu\left(D\cap B'\right)}{\mu\left(B\right)}\mu\left(C\right)\geq\frac{\epsilon^{2}}{2}\mu\left(C\right).
\end{align*}
This is a contradiction, hence for every $D\in\mathcal{B}_{\Sigma_{A}}$
which is $\mathcal{T}$-invariant either $\mu\left(D\right)=0$ or
$\mu\left(\Sigma_{A}\backslash D\right)=0$. 
\end{proof}
\begin{proof}[Proof of Theorem \ref{thm: TMS}]

By Proposition \ref{prop: derivative condition} and Theorem \ref{thm: double tail TMS},
the shift $\left(\Sigma_{A},\mathcal{B}_{\Sigma_{A}},\mu,T\right)$
satisfies the conditions of Theorem \ref{thm: nonsingular erg} hence
if it is conservative, then it is ergodic. 

\end{proof}

\subsubsection{Relation to a closed relative of question 97 from Rufus Bowen's notebook}

Rufus Bowen has asked the following question.
\begin{problem}
Is there a nonergodic volume preserving $C^{1}$ Anosov diffeomorphism
of $\mathbb{T}^{2}?$
\end{problem}
The following variant of this problem is still open. 
\begin{problem}
Is there a nonergodic, conservative $C^{1}$ Anosov diffeomorphism
of $\mathbb{T}^{2}?$
\end{problem}
Note that if the diffeomorphism is $C^{1+\alpha}$ for $\alpha>0$
then by \cite{GurOse73} being conservative is equivalent to having
an absolutely continuous invariant measure, while by \cite{Kos16Anosov}
this is no longer true in the $C^{1}$ category. A natural approach
for this problem is to start with an hyperbolic diffeomorphism of
$\mathbb{T}^{2}$ with a nice Markov partition which gives a topological
semiconjugacy $\Theta:\left(\Sigma_{A},T\right)\to\left(\mathbb{T}^{2},f\right)$.
The push-forward by $\Theta$ of the class of inhomogeneous Markov
shifts on the symbolic space is a natural class of nonsingular measures
for $f$. In \cite{Kos16Anosov}, the examples were constructed by
a smooth realization process of such a Markov measure. These measures
are nice for that realization scheme as the Markov property enables
one to build the realization by an iterated scheme. Theorem \ref{thm: double tail TMS}
shows that this class of inhomogeneous Markov measures under a natural
condition are either dissipative or ergodic. 

\subsection{Bernoulli shifts on groups with the ratio ergodic theorem property}

A countable group $G$ satisfies the Ratio-Ergodic-Theorem (RET) property
if there exists an increasing sequence of finite subsets $F_{n}\subset G$,
$\bigcup_{n}F_{n}=G$ such that for any non-singular, conservative
$G$-action $G\curvearrowright\left(X,\mathcal{B},\mu\right)$ with
$\mu\left(X\right)=1$, for all $f\in L^{1}\left(X,\mathcal{B},\mu\right)$
and for $\mu$ almost every $x\in X$, 
\[
R_{n}\left(f,1\right)(x):=\frac{\sum_{g\in F_{n}}\frac{d\mu\circ T_{g}}{d\mu}(x)f\circ T_{g}(x)}{\sum_{g\in F_{n}}\frac{d\mu\circ T_{g}}{d\mu}(x)}\xrightarrow[n\to\infty]{}h(f,1)(x)
\]
where $h=h\left(f,1\right)$ satisfies 
\begin{itemize}
\item If $f\geq0$, then $h\geq0$. 
\item For all $g\in G$, $h\circ T_{g}=h$.
\item $\int_{X}hkd\mu=\int_{X}k1_{A}d\mu$ for all $k\in L^{\infty}\left(X,\mu\right)$
satisfying for all $g\in G$, $k\circ T_{g}=k$.
\end{itemize}
We say that $G$ satisfies the Hurewicz Maximal inequality property,
henceforth abbreviated as $G$ is a Hurewicz group, if in addition
there exists $C>0$ such that for all $f\in L^{1}\left(X,\mathcal{B},\mu\right)$
and $\epsilon>0$,
\[
\mu\left(\sup_{n\in\mathbb{N}}\left|R_{n}(f,1)\right|>\epsilon\right)\leq C\frac{|f|_{1}}{\epsilon}.
\]
 Examples of Hurewicz RET groups are $\mathbb{Z}$ (Hurewicz's theorem),
$\mathbb{Z}^{d}$ \cite{Hoc1} with $F_{n}=[-n,n]^{d}$ and discrete
Heisenberg groups $H^{d}\left(\mathbb{Z}\right)$ \cite{Jarret17}.
Hochman has shown a connection between the Hurewicz property for amenable
groups and existence of Følner sequences which satisfy the Besicovitch
covering property. See \cite{Hoc2} for these definitions and precise
statements. 

Given a countable group $G$ and $N\in\mathbb{N}$, the Bernoulli
action of $G$ on $\{1,..,N\}^{G}$ is defined by 
\[
\left(T_{g}(x)\right)_{h}=x_{g^{-1}h}.
\]
By Kakutani's dichotomy \cite[P. 528, Thm 3]{Shiryaev_prob}, $\mu\circ T_{g}$
and $\mu$ are equivalent if and only if 
\begin{equation}
\sum_{h\in G}\sum_{j=1}^{N}\left(\sqrt{\mu_{h}\left(j\right)}-\sqrt{\mu_{g^{-1}h}\left(j\right)}\right)^{2}<\infty.\label{eq: Kakutani-1}
\end{equation}
Therefore, the shift is nonsingular if and only if for all $g\in G$,
equation (\ref{eq: Kakutani-1}) holds. 
\begin{thm}
\label{thm: Bernoulli groups}Let $G$ be a countable Hurewicz RET
group. If a non-singular Bernoulli shift\\
 $\left(\left\{ 1,..,N\right\} ^{G},\mathcal{B},\prod_{g\in G}\mu_{g},\left(T_{g}\right)_{g\in G}\right)$
is conservative and 
\begin{equation}
L=\sup_{g\in G}\frac{\max_{j\in\{1,...,N\}}\left(\mu_{g}\left(\{j\}\right)\right)}{\min_{j\in\{1,...,N\}}\left(\mu_{g}\left(\{j\}\right)\right)}<\infty,\label{eq: condition on B-shifts-1}
\end{equation}
then it is ergodic.
\end{thm}
The proof of this theorem is identical to the proof of Theorem \ref{thm: ergodicity of Bernoulli shifts}
once one replaces the relation $\mathcal{T}$ with the $\left(F_{n}\right)_{n}$
homoclinic relation
\[
\mathcal{HOM}=\left\{ (x,y)\in\{1,..,N\}^{G}:\ \exists n\in\mathbb{N},\ x|_{G\backslash F_{n}}=y|_{G\backslash F_{n}}\right\} .
\]
Here $F_{n}$ is the sequence from the definition of Hurewicz-RET
group. The use of the Ratio and Maximal ergodic theorems is similar. 
\begin{problem*}
Is Theorem \ref{thm: Bernoulli groups} true for a general countable
amenable group $G$?
\end{problem*}
This problem is interesting as there are currently very few groups
which are known to be RET and Hurewicz. Hochman has shown that there
are abelian groups (hence amenable) which are not RET. 

\bibliographystyle{plain}
\bibliography{biblioNS}

\end{document}